\documentclass[11pt]{amsart}
\usepackage{a4wide,amssymb}

\newcommand{\w}{\omega}
\newcommand{\IQ}{\mathbb Q}
\newcommand{\IS}{\mathbb S}
\newcommand{\IN}{\mathbb N}
\newcommand{\e}{\varepsilon}
\newcommand{\F}{\mathcal F}
\newcommand{\U}{\mathcal U}

\newtheorem{theorem}{Theorem}
\newtheorem{corollary}[theorem]{Corollary}

\newtheorem{lemma}[theorem]{Lemma}
\theoremstyle{definition}
\newtheorem{remark}[theorem]{Remark}

\title{On the spread of topological groups containing subsets of the Sorgenfrey line}
\author{Taras Banakh, Igor Guran, Alex Ravsky}
\dedicatory{Dedicated to the 60th birthday of M.M. Zarichnyi}
\address{T.Banakh: Ivan Franko National University of Lviv (Ukraine) and Institute of Mathematics, Jan Kochanowski University in Kielce (Poland)}
\email{t.o.banakh@gmail.com}
\address{I. Guran:  Ivan Franko National University of Lviv, Ukraine}
\email{igor-guran@ukr.net }
\address{A.~Ravsky: Institute for Applied Problem of Mechanics and Mathematcis of Unkrainian Academy of Sciences, Lviv, Naukova 3b, Ukraine}
\email{alexander.ravsky@uni-wuerzburg.de}

\begin{document}
\begin{abstract} We prove that any topological group $G$ containing a subspace $X$ of the Sorgenfrey line has spread $s(G)\ge s(X\times X)$. Under OCA, each topological group containing an uncountable subspace of the Sorgenfrey line has uncountable spread. This implies that under OCA a cometrizable topological group $G$ is cosmic if and only if it has countable spread. On the other hand, under CH there exists a cometrizable Abelian topological group that has hereditarily Lindel\"of countable power and contains an uncountable subspace of the Sorgenfrey line. This cometrizable topological group has countable spread but is not cosmic.
\end{abstract}
\maketitle

\section*{Introduction}

The main result of this paper is the following theorem answering the problem \cite{MO1}, posed by the first author on MathOverflow.

\begin{theorem}\label{t:main1} Each topological group containing a topological copy of the Sorgenfrey line contains a discrete subspace of cardinality continuum. 
\end{theorem}

We recall that {\em the Sorgenfrey line} is the real line endowed with the topology, generated by the half-intervals $[a,b)$ where $a<b$ are arbitrary real numbers. The Sorgenfrey line endowed with the (continuous) operation of addition of real numbers is a classical example of a paratopological group, which is not a topological group, see \cite[1.2.1]{AT}. The Sorgenfrey line has countable spread and shows that Theorem~\ref{t:main1} cannot be generalized to paratopological groups.

Theorem~\ref{t:main1} follows from a more refined theorem evaluating the spread of a topological group that contains a topological copy of an uncountable subspace of the Sorgenfrey line.

We recall that for a topological space $X$ the cardinal
$$s(X)=\sup\{|D|:D\subset X\mbox{ is a discrete subspace of }X\}$$is called the  
 {\em spread} of $X$.
 
 \begin{theorem}\label{t:main} Assume that a topological group $G$ contains a subspace $X$, homeomorphic to an uncountable subspace of the Sorgenfrey line. Then $s(G)\ge s(X\times X)$.
 \end{theorem} 
 
 Theorems~\ref{t:main1} and \ref{t:main} will be proved in Section~\ref{s:main}. 
Theorem~\ref{t:main} has the following corollary holding under OCA (the Open Coloring Axiom, see \cite[\S8]{Tod}).

\begin{corollary}\label{c:OCA} Under OCA any topological group $G$ containing an uncountable subspace $X$ of the Sorgenfrey line  has uncountable spread.
\end{corollary}

\begin{proof} Proposition 8.4(c) of \cite{Tod} implies that $X$ contains an uncountable subset $Z$ admitting a strictly decreasing function $f:Z\to X$ (with respect to the linear order inherited from the real line). Then $D=\{(x,f(x)):z\in Z\}$ is a discrete subspace of $X\times X$ and hence $s(G)\ge s(X\times X)\ge|D|=|Z|>\omega$.
\end{proof}

We shall apply Corollary~\ref{c:OCA} to detect cosmic topological groups among cometrizable topological groups.

A topological space $X$ 
\begin{itemize}
\item is {\em cosmic} if it is a continuous image of a separable metrizable space;
\item is {\em cometrizable} if $X$ admits a weaker metrizable topology such that each point has a (not necessarily open) neighborhood base consisting of sets which are closed in the metric topology. 
\end{itemize}

Cometrizable spaces were introduced by Gruenhage in \cite{Gru89}. The interplay between cometrizable spaces and other generalized metric spaces was studied in \cite{MO2} and \cite{BRS}. It is is proved in \cite{MO2} and \cite{BRS} that the class of cometrizable spaces includes all stratifiable and all sequential $\aleph_0$-spaces; on the other hand, there exists a countable (and hence cosmic) space, which is not cometrizable.

  In \cite{Gru89} Gruenhage proved that under PFA a regular cometrizable space $X$ is cosmic if and only if $X$ has countable spread and contains no uncountable subspace of the Sorgenfrey line. In \cite[8.5]{Tod} Todor\v cevi\'c observed that this characterization remains true under OCA (which is a weaker assumption than PFA).
Unifying Theorem 8.5 \cite{Tod} of Todor\v cevi\'c with Corollary~\ref{c:OCA}, we obtain the following OCA-characterization of cosmic topological groups.

\begin{corollary}\label{c:OCA2} Under OCA, a cometrizable topological group is cosmic if and only if it has countable spread.
\end{corollary}

It is interesting that this OCA-characterization of cosmic cometrizable groups does not hold under the Continuum Hypothesis (briefly, CH).

\begin{theorem}\label{t:CH} Under CH there exists a cometrizable topological group $G$ that contains an uncountable subspace of the Sorgenfrey line (and hence is not cosmic) but has hereditarily Lindel\"of countable power $G^\w$ (and hence $G^\w$ has countable spread).
\end{theorem}

Theorem~\ref{t:CH} will be proved in Section~\ref{s:CH}.

\begin{remark} By \cite{PW}, there exists a hereditarily Lindel\"of topological group $G$ whose square is not normal. The topological group $G$ has countable spread but is not cosmic. Corollary~\ref{c:OCA2} implies that the space $G$ is not cometrizable under OCA.
\end{remark}

\begin{remark} Using the Continuum Hypothesis, Hajnal and Juh\'asz \cite{HJ} constructed a hereditarily separable Boolean topological group $G$ with uncountable pseudocharacter. This topological group has countable spread (being hereditarily separable) but is not hereditarily Lindel\"of and not cosmic (because it has uncountable pseudocharacter).
\end{remark}



\section{Proof of Theorem~\ref{t:main}}\label{s:main}

Theorems~\ref{t:main1} and \ref{t:main} will be deduced from the following 

\begin{lemma}\label{l:main} Let $\kappa$ be a cardinal of uncountable cofinality and $X$ be a subspace of the Sorgenfrey line whose square contains a discrete subspace $\Gamma\subset X\times X$ of cardinality $|\Gamma|=\kappa$. If a topological group $G$ contains a subspace homeomorphic to $X$, then $G$ contains a discrete subspace of cardinality $\kappa$.
\end{lemma}

\begin{proof} We shall identify the subspace $X$ of the Sorgenfrey line with a subspace of the topological group $G$. For every $x\in X$ and a rational number $q>x$ let $[x,q)=\{y\in X:x\le y<q\}$ be the order half-interval in $X$. Let also ${\uparrow}x=\{y\in X:x\le y\}$. 
By the definition of the Sorgenfrey topology, the countable family $\{[x,q):x<q\in\mathbb Q\}$ is a neighborhod base at $x$ in the space $X$.

Since the subspace $\Gamma\subset X\times X$ is discrete, each point $(x,y)\in \Gamma$ has a neighborhood $O_{(x,y)}\subset X\times X$ such that $\Gamma\cap O_{(x,y)}=\{(x,y)\}$. Find rational numbers $u_{(x,y)},v_{(x,y)}$ such that $(x,y)\in [x,u_{(x,y)})\times [y,v_{(x,y)})\subset O_{(x,y)}$. Since the cardinal $|\Gamma|=\kappa$ has uncountable cofinality, for some rational numbers $u,v$ the set $\Gamma'=\{(x,y)\in \Gamma:u_{(x,y)}=u,\;v_{(x,y)}=v\}$ has cardinality $|\Gamma'|=|\Gamma|$. Replacing the set $\Gamma$ by the set $\Gamma'$, we can assume that $u_{(x,y)}=u$ and $v_{(x,y)}=v$ for all $(x,y)\in \Gamma$.  

Let $\Gamma_1:=\{x\in X:\exists y\in X\;(x,y)\in\Gamma\}$ and $\Gamma_2=\{y\in X:\exists x\in X\;(x,y)\in\Gamma\}$ be the projections of the set $\Gamma\subset X\times X$ onto the coordinate axes.
We claim that $\Gamma$ coincides with the graph of some strictly decreasing function $f:\Gamma_1\to\Gamma_2$. First observe that for any $x\in \Gamma_1$ there exists a unique $y\in\Gamma$ with $(x,y)\in\Gamma$. Otherwise we could find two real numbers $y_1<y_2$ with $(x,y_1),(x,y_2)\in\Gamma$ and conclude that $$(x,y_2)\in [x,u)\times [y_2,v)\subset  [x,u)\times [y_1,v)\subset O_{(x,y_1)},$$
which contradicts the choice of the neighborhood $O_{(x,y_1)}$.
This contradiction shows that $\Gamma$ coincides with the graph of some function $f:\Gamma_1\to \Gamma_2$. Let us show that this function is strictly decreasing. Assuming that this is not true, we could find two points $(x_1,y_1),(x_2,y_2)\in\Gamma$ with $x_1<x_2$ and $y_1\le y_2$. Then 
 $$(x_2,y_2)\in [x_2,u)\times [y_2,v)\subset 
[x_1,u)\times [y_1,v)\subset O_{(x_1,y_1)},$$
which contradicts the choice of the neighborhood $O_{(x_1,y_1)}$.

Therefore the function $f:\Gamma_1\to \Gamma_2$ is strictly decreasing, which implies that $|\Gamma_1|=|\Gamma_2|=|\Gamma|=\kappa$.
 
For any point $x\in X$ choose a neighborhood $V_x\subset G$ of the unit $e$ of $G$ such that $$X\cap (V_x^{-1}V_xx\cup xV_xV_x^{-1})\subset {\uparrow}x.$$
Next, for every point $x\in X$, choose a rational point $r_x>x$ such that $[x,r_x)\subset xV_{f(x)}$ if $x\in \Gamma_1$ and $[x,r_x)\subset V_{f^{-1}(x)}x$ if $x\in \Gamma_2$. Since the cardinal $|\Gamma_1|=\kappa$ has uncountable cofinality, for some $c,d\in \IQ$ the set $Z=\{z\in \Gamma_1:r_z=c,\; r_{f(z)}=d\}$ has cadinality $\kappa$. 

We claim that the subspace $D:=\{z\cdot f(z):z\in Z\}$ has cardinality $\kappa$ and is discrete in $G$. For every $z\in Z$ consider the neighborhood $z(V_z\cap V_{f(z)})f(z)$ of the point $z\cdot f(z)$ in $G$. We claim that  $x\cdot f(x)\notin z(V_z\cap V_{f(z)})f(z)$ for any $x\in Z\setminus\{z\}$. To derive a contradiction, assume that $x\cdot f(x)\in zV_{f(z)}f(z)$ for some $x\ne z$ in $Z$.

If $x>z$, then $x\in [z,r_z)\subset zV_{f(z)}$ and $$f(x)=x^{-1}xf(x)\in x^{-1}zV_{f(z)}f(z)\subset V_{f(z)}^{-1}z^{-1}zV_{f(z)}f(z)=V_{f(z)}^{-1}V_{f(z)}f(z).$$Then $f(x)\in X\cap V_{f(z)}^{-1}V_{f(z)}f(z)\subset {\uparrow}f(z)$ and $f(x)\ge f(z)$, which is not possible as $x>z$ and $f$ is strictly decreasing.

If $z>x$, then $f(x)>f(z)$ and $f(x)\in [f(x),r_{f(x)})=[f(x),d)\subset [f(z),d)= [f(z),r_{f(z)})\subset V_{z}f(z)$ and then
$$x\in zV_zf(z)f(x)^{-1}\subset zV_zf(z)f(z)^{-1}V_z^{-1}=zV_zV_z^{-1}\subset{\uparrow}z$$ which contradicts $z>x$.
  \end{proof}
\smallskip

\noindent{\em Proof of Theorem~\ref{t:main1}.} Assume that a topological group $G$ contains a topological copy of the Sorgenfrey line $\mathbb S$. Observe that the square of $\mathbb S$ contains a discrete subset $\Gamma=\{(x,-x):x\in\mathbb S\}$ of cardinality continuum $\mathfrak c$. By \cite[5.12]{Jech}, the continuum has uncountable cofinality. Applying Lemma~\ref{l:main},  we conclude that the topological group $G$ contains a discrete subspace of cardinality $\mathfrak c$. \hfill $\square$
\smallskip

\noindent{\em Proof of Theorem~\ref{t:main}.} Let $G$ be a topological group $G$ containing a subspace $X$, homeomorphic to an uncountable subspace of the Sorgenfrey line. Assuming that $s(G)<s(X\times X)$, we conclude that $s(X\times X)\ge \kappa^+$ for the cardinal $\kappa=s(G)$. Then $X\times X$ contains a discrete subspace $D$ of cardinality $|D|=\kappa^+$, which has uncountable cofinality. In this case we can apply Lemma~\ref{l:main} and conclude that $G$ contains a discrete subspace of cardinality $\kappa^+$, which implies that $\kappa=s(G)\ge \kappa^+>\kappa$ and this is a desired contradiction.\hfill $\square$

\section{Proof of Theorem~\ref{t:CH}}\label{s:CH}

In this section we prove Theorem~\ref{t:CH}. But first we prove that the Sorgenfrey line $\IS$ embeds into a cometrizable topological group. In the proof of this embedding result, we use the $k$-separability of $\IS$.

A subset $D$ of a topological space $X$ is called {\em $k$-dense} in $X$ if each compact subset $K\subset X$ is contained in a compact set $\tilde K\subset X$ such that the intersection $D\cap\tilde K$ is dense in $\tilde K$. 

A topological space $X$ is defined to be {\em $k$-separable} if it contains a countable $k$-dense subset. 

\begin{lemma}\label{l:Qd} The set $\IQ$ of rational numbers is $k$-dense in the Sorgenfrey line $\mathbb S$.
\end{lemma}

\begin{proof} Given a compact set $K\subset\IS$, observe that $K$ is metrizable and hence contains a countable dense subset $\{x_n\}_{n\in\w}\subset K$. For every $n,k\in\w$ fix a  rational numbers $x_{n,k}$ such that $x_n<x_{n,k}<x_n+\frac1{2^{n+k}}$. We claim that the subset $\tilde K=K\cup\{x_{n,k}\}_{n,k\in\w}$ is compact. Indeed, let $\mathcal U$ be a cover of $\tilde K$ by open subsets of $\IS$. For every $x\in K$ find a set $U_x\in\mathcal U$ with $x\in U_x$ and a  real number $b_x$ such that $[x,b_x)\subset U_x$. By the compactness of $K$ the open cover $\{[x,b_x):x\in K\}$ of $K$ has a finite subcover $\{[x,b_x):x\in F\}$ (here $F$ is a suitable finite subset of $K$). For every $x\in F$ the set $[x,b_x)$ is closed in $\IS$ and hence the intersection $K\cap [x,b_x)$ is compact, which implies that the number $\e_x:=b_x-\max \big(K\cap[x,b_x)\big)$ is strictly positive. Choose $m\in\IN$ such that $\frac1{2^m}<\min_{x\in F}\e_x$. Then $\tilde K\setminus\bigcup_{x\in F}[x,b_x)\subset \{x_{n,k}:n+k\le m\}$ is finite and hence is contained in the union $\bigcup\mathcal F$ of some finite subfamily $\F\subset \U$. Then $\F\cup\{U_x:x\in F\}\subset\U$ is a finite subcover of $\tilde K$, witnessing that the subset $\tilde K$ of $\IS$ is compact. By the definition of $\tilde K$, the set $\tilde K\cap \IQ\supset\{x_{n,k}\}_{n,k\in\w}$ is dense in $\tilde K$.
\end{proof} 

Lemma~\ref{l:Qd} implies that the Sorgenfrey line is $k$-separable. Now we prove that for any $k$-separable space $X$ and a cometrizable space $Y$ the function space $C_k(X,Y)$ is cometrizable.
Here for topological spaces $X,Y$ by $C_k(X,Y)$ we denote the space of continuous functions from $X$ to $Y$, endowed with the compact-open topology, which is generated by the subbase consisting of the sets
$$[K,U]:=\{f\in C_k(X,Y):f(K)\subset U\}$$where $K$ is a compact subset of $X$ and $U$ is an open subset of $Y$.

\begin{lemma}\label{l:Ck} For any $k$-separable space $X$ and any cometrizable space $Y$ the function space $C_k(X,Y)$ is cometrizable.
\end{lemma}

\begin{proof} Let $D$ be a countable $k$-dense set in $X$ and $\tau$ be a metrizable topology on $Y$, witnessing that the space $Y$ is cometrizable.
By $Y_\tau$ we denote the metrizable topological space $(Y,\tau)$.

The density of the set $D$ in $X$ ensures that the restriction operator
$$r:C_k(X,Y)\to Y_\tau^D,\;\;r:f\mapsto f{\restriction}D,$$
is injective. Let $\sigma$ be the (metrizable) topology on $C_k(X,Y)$ such that the map $$r:(C_k(X,Y),\sigma)\to Y_\tau^D$$ is a topological embedding. We claim that the topology $\sigma$ witnesses that the space $C_k(X,Y)$ is cometrizable.

Fix any function $f\in C_k(X,Y)$ and an open neighborhood $O_f\subset C_k(X,Y)$. Without loss of generality, $O_f$ is of basic form $O_f=\bigcap_{i=1}^n[K_i,U_i]$ for some non-empty compact sets $K_1,\dots,K_n\subset  X$ and some open sets $U_1,\dots,U_n\subset Y$. For every $i\le n$ and point $x\in K_i$, find a neighborhood $V_{f(x)}\subset Y$ of $f(x)\in U_i$ whose $\tau$-closure $\overline{V}^\tau_{f(x)}$ is contained in $U_i$. Using the regularity of the cometrizable space $Y$, find two open neighborhoods $N_{f(x)},W_{f(x)}$ of $f(x)$ such that $\overline{N}_{f(x)}\subset W_{f(x)}\subset\overline{W}_{f(x)}\subset V_{f(x)}$.

By the compactness of $K_i$, the open cover $\{f^{-1}(N_{f(x)}):x\in K_i\}$ of $K_i$ has a finite subcover $\{f^{-1}(N_{f(x)}):x\in F_i\}$ where $F_i\subset K_i$ is a finite subset of $K_i$. By the $k$-density of $D$ in $X$, for every $x\in F_i$ the compact set $K_{i,x}:=K_i\cap f^{-1}(\bar N_{f(x)})$ can be enlarged to a compact set $\tilde K_{i,x}\subset X$ such that $K_{i,x}$ is contained in the closure of the set $\tilde K_{i,x}\cap D$. Replacing the set $\tilde K_{i,x}$ by $\tilde K_{i,x}\cap f^{-1}(\overline{W}_{f(x)})$, we can assume that $f(\tilde K_{i,x})\subset\overline{W}_{f(x)}\subset V_{f(x)}$.

Consider the open neighborhood $$V_f=\bigcap_{i=1}^n\bigcap_{x\in F_i}[\tilde K_{i,x},V_{f(x)}]$$of $f$ in the function space $C_k(X,Y)$. We claim that its $\sigma$-closure $\overline{V}^\sigma_f$ is contained in $O_f$. 

Given any function $g\notin O_f$, we should find a neighborhood $O_g\in\sigma$ of $g$ that does not intersect $V_f$. Since $g\notin O_f$, there exists $i\le n$ and a point $z\in K_i$ such that $g(z)\notin U_i$. Find a point $x\in F_i$ with $z\in K_{i,x}$. Taking into account that $\overline{V}^\tau_{f(x)}\subset U_i\subset Y\setminus\{g(z)\}$, we conclude that $g(z)\notin \overline{V}^\tau_{f(x)}$. Since the point $z$ belongs to the closure of the set $\tilde K_{i,n}\cap D$, the continuity of the function $g:Z\to Y_\tau$ yields a point $d\in \tilde K_{i,n}\cap D$ such that $g(d)\notin  \overline{V}^\tau_{f(x)}$. Then $O_g:=\big[\{d\},Y\setminus \overline{V}^\tau_{f(x)}\big]\in\sigma$ is a required $\sigma$-open neighborhood of $g$ that is disjoint with the neighborhood $V_f$.
\end{proof} 

\begin{lemma}\label{l:emb} The Sorgenfrey line $\IS$ admits a topological embedding into the cometrizable locally convex linear vector space $C_k(\IS)$.
\end{lemma}

\begin{proof} By Lemma~\ref{l:Qd}, the Sorgenfrey line $\IS$ is $k$-separable, and by Lemma~\ref{l:Ck}, the function space $C_k(\IS)$ is cometrizable. It remains to observe that the map $\chi:\IS\to C_k(\IS)$ assigning to each point $x\in\IS$ the function $\chi_x:\IS\to\{0,1\}$ defined by $\chi_x^{-1}(1)=[-x,\infty)$ is a topological embedding of $\IS$ into  the function space $C_k(\IS)$, which has the structure of a locally convex topological vector space.
\end{proof}

\noindent{\em Proof of Theorem~\ref{t:CH}}. By Lemma~\ref{l:emb}, the Sorgenfrey line $\IS$ can be identified with a subspace of some cometrizable Abelian topological group $H$. According to a result of Michael \cite{Mich}, under CH the Sorgenfrey line contains an uncounatble subspace $X$ whose countable power  $X^\w$ is hereditarily Lindel\"of. Observe that the topological sum $X^{<\w}=\oplus_{n\in\w}X^n$ of finite powers of $X$ admits a topological embedding into $X^\w$, which implies that $X^{<\w}$ is hereditarily Lindel\"of as well as its countable power $(X^{<\w})^\w$. 

Observing that the group hull $G$ of $X$ in the group $H\supset\IS\supset X$ is a continuous image of $X^{<\w}$, we conclude that the space $G$ is hereditarily Lindel\"of. Moreover,  the countable power $G^\w$ is hereditarily Lindel\"of, being a continuous image of the hereditarily Lindel\"of space $(X^{<\w})^\w$.\hfill $\square$

\end{document}